\documentclass[microtype]{gtpart}
\usepackage[all]{xy}
\usepackage[parfill]{parskip}
\usepackage{amsmath}
\usepackage{amscd, latexsym, hyperref, rlepsf, times}
\usepackage{color,graphicx}
\newtheorem{dummy}{anything}[section]
\newtheorem{theorem}[dummy]{Theorem}

\newtheorem{corollary}[dummy]{Corollary}

\theoremstyle{definition}

\newtheorem{remark}[dummy]{Remark}

\renewcommand
{\theequation}{\thesection.\arabic{dummy}}

\newcommand{\bfc}{{\mathbb{C}}}
\newcommand{\OB}{\mathcal{OB}}
\newcommand{\rank}{\operatorname{rk}}
\newcommand{\bfz}{{\mathbb{Z}}}
\def\v{\vskip.12in}
\begin{document}
\title[]{\ \ \ \  Milnor fillable contact structures are universally tight}

\author[Lekili and Ozbagci]{Yank\i\ Lekili and Burak Ozbagci}



\begin{abstract} 

We show that the canonical contact structure on the link of a normal complex
singularity is universally tight. As a corollary we show the existence of closed,
oriented, atoroidal $3$-manifolds with infinite fundamental groups which carry
universally tight contact structures that are not deformations of taut (or Reebless)
foliations.  This answers two questions of Etnyre in \cite{etn}. 

\end{abstract}
\maketitle
\section{Introduction}

Let $(X,x)$ be a normal complex surface singularity. Fix a local embedding of $(X,x)$
in $(\bfc^N, 0)$. Then a small sphere $S^{2N-1}_{\epsilon} \subset \bfc^N$ centered at
the origin intersects $X$ transversely, and the complex hyperplane distribution
$\xi_{can}$ on $M=X\cap S^{2N-1}_{\epsilon}$ induced by the complex structure on $X$
is called the \emph{canonical}  contact structure.  For sufficiently small radius
$\epsilon$, the contact manifold  is independent of $\epsilon$ and the embedding, up
to isomorphism. The $3$-manifold $M$ is called the link of the singularity,  and $(M,
\xi_{can})$ is called the \emph{contact boundary} of $(X,x)$. 

A contact manifold $(Y, \xi)$ is said to be \emph{Milnor fillable} if it is isomorphic
to the contact boundary $(M, \xi_{can})$ of some isolated complex  surface singularity
$(X, x)$.  In addition, we say that a closed and oriented $3$-manifold $Y$ is Milnor
fillable if it carries a contact structure $\xi$ so that $(Y,\xi)$ is Milnor fillable.
It is known that  a closed and oriented $3$-manifold is Milnor fillable if and only if
it can be obtained by plumbing according to a weighted graph with negative definite
intersection matrix (cf. \cite{mum} and \cite{gr}).  Moreover any $3$-manifold has at
most one Milnor fillable contact structure up to \textit{isomorphism} (cf.
\cite{cnp}). Note that Milnor fillable contact structures are Stein fillable  (see
\cite{bd}) and hence tight \cite{eg}.  Here we prove that every Milnor fillable
contact structure is in fact universally tight, i.e., the pullback to the universal
cover is tight.  We would like to point out that universal tightness of a contact
structure is not implied by any other type of fillability.

In \cite{etn}, Etnyre settled  a question of Eliashberg and Thurston \cite{elit} by
proving that every contact structure on a closed oriented $3$-manifold is obtained by
a deformation of a foliation and raised two other related questions: 

\v

(Question $4$ in \cite{etn})  \emph{Is every
universally tight contact structure on a closed $3$-manifold with infinite
fundamental group the deformation of a Reebless foliation?}  

\v

(Question $5$ in \cite{etn})  \emph{Is every
universally tight contact structure on an atoroidal closed $3$-manifold with infinite
fundamental group the deformation of a taut foliation?} 

\v
In this note we answer both questions negatively as a consequence of our main result, although one does not necessarily need our main result to find counterexamples. As a matter of fact, one can drive the same consequence by  the existence of  (small) Seifert fibered $L$-spaces carrying transverse contact structures  which are known to be universally tight (see Remark~\ref{ibne}).

The assumption on the fundamental group is necessary since every foliation on a closed $3$-manifold with finite fundamental group has a Reeb component (and hence is not taut) by a theorem of Novikov.  Moreover Ghiggini \cite{gh} gave examples of toroidal $3$-manifolds which carry universally tight contact structures that are not weakly fillable (and therefore can not be perturbations of taut foliations by \cite{elit}). 


We contrast our result with the result of Honda, Kazez and Mati\'{c}  in \cite{HKM}, where
they show that for a sutured manifold  with annular sutures, the existence
of a (universally) tight contact structure is equivalent to the existence of a taut
foliation.


We assume that all the $3$-manifolds are compact and oriented,  all the contact
structures are co-oriented and positive and all the surface singularities are isolated
and normal. 

\section{Milnor fillable implies universally tight}

A {\em graph manifold} is a 3-manifold $M(\Gamma)$ obtained by plumbing circle bundles
according to a connected weighted plumbing graph $\Gamma$. More precisely, let $A_1,
\ldots, A_r$ denote vertices of a connected graph $\Gamma$. Each vertex is decorated
with a pair $(g_i, e_i)$ of integral weights, where $g_i \geq 0$.  Here the $i$th
vertex represents an oriented circle bundle of Euler number $e_i$ over a closed
Riemann surface of  genus $g_i$. Then $M(\Gamma)$ is the $3$-manifold obtained by
plumbing these circle bundles according to $\Gamma$. This means that if there is an edge 
connecting  two vertices in $\Gamma$,  then one glues the circle bundles corresponding to these vertices as follows.  
First one removes a neighborhood of a circle fibre on each circle bundle which is given by the preimage of a disk on the base. 
The resulting  boundary torus on each circle bundle can be identified with  $S^1 \times S^1$ using the natural trivialization of the circle fibration
over the disk that is removed. Now one glues these bundles together using the diffeomorphism that exchanges the two circle
factors on the boundary tori.

A \emph{horizontal} open book in
$M(\Gamma)$ is an open book whose binding consists of some fibers in the circle
bundles and whose (open) pages are transverse to the fibers.  We also  require that the
orientation induced on the binding by the pages coincides with the orientation of the
fibers induced by the fibration. 

In this paper, we will consider horizontal open books on graph manifolds coming from
isolated normal complex singularities.  Given an analytic function $f\colon (X,x)
\rightarrow (\bfc , 0)$ vanishing at $x$, with an isolated singularity at $x$, the
open book decomposition $\OB_f$ of the boundary $M$ of $(X,x)$ with binding $L=M \cap
f^{-1} (0)$ and projection $ \pi=\frac{f}{|f|}\colon M \setminus L \to S^1 \subset
\bfc$ is called the \emph{Milnor open book} induced by $f$.

\begin{theorem} \label{unit} A Milnor fillable contact structure is universally tight. 
 
\end{theorem}

\begin{proof}  Given a Milnor fillable contact $3$-manifold $(Y, \xi)$. By definition
$(Y, \xi)$  is isomorphic to the link $(M, \xi_{can})$ of some surface singularity.
Hence it suffices to show that $(M, \xi_{can})$ is universally tight. It is known that
$M$ is an irreducible graph manifold $M(\Gamma)$ where $\Gamma$ is a negative definite
plumbing graph \cite{neumann}.  Moreover, such a manifold is characterized by the
property that there exists a unique minimal set  $\mathcal{T} $ (possibly empty)  consisting of pairwise disjoint \emph{incompressible} tori in $M$ 
such that each component of $M - \mathcal{T}$ is an orientable Seifert fibered manifold
with an orientable base \cite{neumann}. In terms of the plumbing description
$\mathcal{T}$ is a subset of the tori that are used to glue the circle bundles in the
definition of $M(\Gamma)$. The set $\mathcal{T}$ is minimal if in plumbing of two circle bundles the homotopy class of circle fiber in one boundary torus is not identified with the homotopy class of the fiber in the other boundary torus. 

Recall that an arbitrary Milnor open book $\OB$ on $M$ has the following essential features \cite{cnp}: It is compatible with 
the canonical contact structure $\xi_{can}$, horizontal when restricted to each
Seifert fibered piece in $M - \mathcal{T}$ which means that the Seifert fibres
intersect the pages of the open book transversely, and the binding of the open book consists of some number (which we can take to be non-zero) of regular fibres of the Seifert fibration in each Seifert fibred piece.

In the rest of the proof, we will construct a universally tight contact structure $\xi$ on $M$ which is compatible with the Milnor open book $\OB$.  This implies that 
the canonical contact structure $\xi_{can}$ is isotopic to $\xi$  (since they are both compatible with $\OB$)  and thus we conclude that $\xi_{can}$ on the singularity link $M$ is universally tight. 

Let $V_i$ denote a Seifert fibered $3$-manifold with boundary, which is a component of   $M - N(\mathcal{T})$, where  $N(\mathcal{T})$ denotes a regular neighborhood of $\mathcal{T}$.  Consider the $3$-manifold  $V_i^\prime$ obtained by removing a regular neighborhood of the binding of $\OB$ from $V_i$.  Note that  $V_i^\prime$ is also a Seifert fibered manifold since the binding consists of regular fibers of the Seifert fibration on $V_i$. Then the restriction of a page of $\OB$  to $V_i^\prime$ is a connected horizontal surface  (see the proof of Proposition $4.6$ in \cite{cnp}) which we denote by $\Sigma_i^\prime$.  It follows that $V_i^\prime$ is a surface bundle over $S^1$ whose fibers are precisely the restriction of the pages of $\OB$ to $V_i^\prime$, since $\Sigma_i^\prime$ does not separate $V_i^\prime$. Note that $\Sigma_i^\prime$ is a  branched cover 
of the base of the Seifert fibration  on $V_i^\prime$ and the monodromy $\phi_i$ of this surface bundle 
is a periodic self-diffeomorphism of $\Sigma_i^\prime$ of some order $n_i$ (cf. Section $1.2$ in \cite{H}). 

Now we construct, as in Section $2$ in \cite{gh},  a contact structure $\xi_i^\prime$
on $V_i^\prime$  which is ``compatible'' with the surface fibration $V_i^\prime \to
S^1$.  Here compatibility means that the Reeb vector field of the contact form  is
transverse to the fibers, keeping in mind that a fiber of this fibration is cut out
from a page of the open book $\OB$. Let $\beta_i$ denote a $1$-form on
$\Sigma_i^\prime$ such that  $d\beta_i$ is a volume form on $\Sigma_i^\prime$ and
$\beta_i|_{\partial \Sigma_i^\prime}$ is a volume form on $\partial \Sigma_i^\prime$.
Then the $1$-form $$ \beta_i^\prime = \frac{1}{n_i} \displaystyle\sum_{k=0}^{n_i -1}
(\phi_i^k)^* \beta_i ,$$ which also satisfies the above conditions, is a  $\phi_i$
invariant $1$-form on $\Sigma_i^\prime$. Let $t$ denote the coordinate on $S^1$.  It
follows that  for every real number $\epsilon > 0$, the kernel of the $1$-form $dt +
\epsilon \beta_i^\prime$ is a contact structure  on $V_i^\prime$ which is compatible
with the fibers. Note that the characteristic foliation on every torus in $\partial
V_i^\prime$ is linear  with a slope arbitrarily close to the slope of the foliation
induced by the pages when $\epsilon \to 0$.  Here we point out that, for fixed
$\epsilon > 0$, different choices of $\beta_i$ give isotopic contact structures by
Gray's theorem,  while the choice of $\epsilon$ will not play any role in our
construction as long as it is sufficiently small. Therefore, we will fix a sufficiently
small $\epsilon$ and denote the isotopy type of this contact structure by
$\xi_i^\prime$.  Moreover the Reeb vector field $R_i$ is tangent to the circle fibers
in the Seifert fibration and hence transverse to the fibers of the surface bundle
$V_i^\prime \to S^1$.   

Furthermore, we observe that $\xi_i^\prime$ is transverse to the Seifert fibration on
$V_i^\prime$ and can be extended over to $V_i$ along the neighborhood of the binding
so that it remains transverse to the Seifert fibration. Now we claim that the resulting
contact structure $\xi_i$ on $V_i$ is universally tight. This essentially follows from
an argument in Proposition $4.4$ in \cite{mas} where the universal tightness of 
transverse contact structures on closed Seifert fibered 3-manifolds is proven (see
also Corollary 2.2 in \cite{lm}). The
difference in our case is that $V_i$ may have toroidal boundary. Nevertheless, the argument
in \cite{mas} still applies. Namely, any contact structure which is transverse to the
fibers of a Seifert manifold (possibly with boundary or non-compact) is
universally tight. Consider first the universal cover of the base of the Seifert
fibration. This can be either $S^2$ or $\mathbb{R}^2$. If it is $S^2$, then the $V_i$
cannot have any boundary, as we arranged that if there is a boundary to $V_i$, it
should be incompressible. Therefore, in that case $\mathcal{T}=\emptyset$ and $M$ is
closed Seifert fibred space with base $S^2$ with a contact structure transverse to the
fibres of the Seifert fibration. The universal cover of $M$ is now obtained by
unwrapping the fibre direction. Hence it is either $S^3$ or $S^2\times \mathbb{R}$
depending on whether $\pi_1(M)$ is finite or infinite. However, it cannot be
$S^2\times \mathbb{R}$ as $M$ is irreducible. In particular, when
$\mathcal{T}=\emptyset$, it
follows that $M$ is either a small Seifert fibered or a lens space and its universal cover is $S^3$. The
contact structure and the Seifert fibration lifts to a transverse contact structure on
$S^3$. It follows that this is the standard tight contact structure on $S^3$ (for
example, see \cite{mas}). Next, suppose that the base of the Seifert fibration on
$V_i$ has universal cover homeomorphic to $\mathbb{R}^2$. We then lift the Seifert
fibration and the contact structure to get a contact structure on $\mathbb{R}^2 \times
S^1$, such that the contact structure is transverse to the $S^1$ factor. Next, we
unwrap the $S^1$ direction to get a contact structure on $\mathbb{R}^2 \times \mathbb{R}$ such
that the contact structure is transverse to the $\mathbb{R}$ factor and invariant under
integral translations in this direction. It follows that this latter contact structure
is the standard tight contact structure on $\mathbb{R}^3$ (see \cite{gir} Section
2.B.c).

Let $V_1, \ldots, V_n$ denote the Seifert fibered manifolds in the decomposition of $M
- N(\mathcal{T})$. Our goal is to glue together  $\xi_i$'s on $V_i$'s to get a
universally tight contact structure  $\xi$ on $M$ which is \emph{compatible} with
$\OB$.  We should point out that if one ignores the compatibility with $\OB$, then
$\xi_i$'s can be glued along the incompressible pre-Lagrangian tori on $\partial
V_i$'s to yield a universally tight contact structure on $M$, by Colin's gluing
theorem \cite{col}.  This was already described in Theorem $1.4$ in \cite{col2},
although the contact structures on Seifert fibered pieces were obtained by perturbing Gabai's taut foliations \cite{g}. 

By construction, the contact structure $\xi_i$ on $V_i$ is compatible with the restriction of $\OB$ to $V_i$.
We first modify  $\xi_i$ near each component of $\partial V_i$ to put it  in a certain
standard form.  To this end, let $N(T_{ij})$ denote the normal neighborhood of a torus
$T_{ij} \in \mathcal{T}$ along which plumbing is performed between $V_i$ and $V_j$. 

Recall that the plumbing was perfomed by trivializing the boundary of the circle
bundles hence identifying them with $T^2=S^1 \times S^1$ and then exchanging the two circle factors.
We can extend these trivialization in a neighborhood of $T_{ij}$, by picking sections
$s_i$ near $T_{ij}$ which extends the section used for the plumbing. Let $r_i$ denote
the fibre direction of the Seifert fibration on $V_i$. Then, we can identify the
boundary of $N(T_{ij})$ in $V_i$ with $T^2$ so that the basis $(r_i,s_i)$ is sent to
the standard basis $\{ \partial_x , \partial_y \} $ of $T^2$. Hence, we can identify $N(T_{ij}) = T^2 \times
[a_i,b_i] \cup_{\rho_{ij}} - T^2\times [a_j,b_j]$ where $\rho_{ij}: T^2 \times \{b_i\}
\to - T^2 \times \{ b_j \} $ is the gluing map used in plumbing sending $(r_i,s_i) \to (s_j,r_j)$.

Let $\mathcal{F}_i$ denote the foliation by circles with a certain rational slope
$m_i/m_j$ on $T^2\times \{a_i\}$ induced by the pages of $\OB$. This means that the
page intersects $T^2 \times \{a_i \}$ at a linear curve tangent to $m_j r_i + m_i s_i$
, we also scale $m_i$ and $m_j$ so that we have $\beta'_i(m_j r_i + m_i s_i)=1$ (The
latter can be arranged as by construction $\beta'_i$ restricts to a volume form on the
boundary of the pages of the open book when restricted to $V_i$). The pages extend
into $T^2 \times [a_i,b_i]$ linearly, as they intersect each $T^2 \times \{c\}$
transversely with slope $m_i/m_j$, thus we obtain the foliation $\mathcal{F}_i \times
[a_i,b_i]$. Similarly, $\mathcal{F}_j$ denote the foliation by circles given by the
intersection of the pages of $\OB$ with $T^2 \times \{a_j\}$ which necessarily has
rational slope $m_j/m_i$ so that the gluing map $\rho_{ij}$ glues the pages in each
piece together to form $\OB$. 

For later convenience, in our identification $N(T_{ij}) = T^2 \times [a_i,b_i]
\cup_{\rho_{ij}} - T^2\times [a_j,b_j]$, we will choose $-\frac{\pi}{2} < a_i < b_i <
\frac{\pi}{2}$ so that $- \cot a_i = m_i/ (m_j - \epsilon) $  is the slope of the
characteristic foliation of the contact structure $\xi_i$ on $T^2\times \{a_i\}$ and
$b_i$ so that $-\cot b_i = m_i /m_j$ is the slope of the pages of $\OB$. By our
construction, the characteristic foliation is the integral of the vector field
$-\epsilon r_i + (m_jr_i+m_i s_i)$ and we can choose $\epsilon$ as small as we need,
so that the slope of the characteristic foliation is arbitrarily close to the slope of
the pages. In particular, we can arrange that $b_i \in (a_i, a_i + \frac{\pi}{2})$.

We now need to glue together the contact forms that we constructed on $V_i$ by extending them to $N(T_{ij})$. For our purposes, we need to pay special attention to compatibility with  $\OB$ on $N(T_{ij})$. 

Consider the contact form $\alpha_i= \cos t dx + \sin t dy$ on  $T^2 \times [a_i,
b_i]$.   By \cite{ch} Lemma $9.1$ we can isotope $\xi_i$ on $V_i$ near the boundary so
that it is defined by a contact form that glue to $\alpha_i$ (note that the slopes of
the characteristic foliations on $T^2 \times \{ a_i \}$ induced by $\xi_i$ and
$\alpha_i$ agree).  Moreover, after this isotopy the Reeb vector field of $\xi_i$ still
remains transverse to the pages of $\OB$ on $V_i$.  Furthermore, the Reeb vector field
of $\alpha_i$, has slope $\tan a_i$ hence it is perpendicular to the slope $- \cot
a_i$ at $T^2 \times \{ a_i \}$ which we know to be arbitrarily close the slope of the
foliation $\mathcal{F}_i \times \{a_i\} $ induced by the page of $\OB$. Since the
slope of the Reeb vector field changes by strictly less than $\pi /2$ as we go from
$a_i$ to $b_i$, the Reeb vector field still remains transverse to $\mathcal{F}_i
\times [a_i,b_i]$.  Therefore, the form $\alpha_i$ is compatible with $\OB$ in $T^2
\times [a_i, b_i]$.  Finally, to finish the construction of the contact structure
$\xi$ on $M$, we observe that the gluing map $\rho_{ij}$ sends $\alpha_i$ to
$\alpha_j$, since we arranged that the slope of $\alpha_i$ and the slope of the
characteristic foliation induced by the page are the same at $T^2 \times \{b_i \}$. 

We constructed a contact structure $\xi$ which is compatible with a Milnor open book (hence is isomorphic to $\xi_{can}$) such that $\xi$ is isotopic to $\xi_i$ on $V_i$, a universally tight contact structure, furthermore for each incompressible torus $T \in \mathcal{T}$, the characteristic foliation of $\xi$ is a linear foliation (with slope $m_i / m_j$). Therefore, we are in a position to apply the gluing result of Colin \cite{col}  which states that universally tight contact structures can be glued along pre-Lagrangian tori to a universally tight contact structure. This shows that $\xi_{can}$ is a universally tight contact structure. \end{proof}

\begin{remark} The above construction shows that when the fibres of each Seifert
fibered piece is not contractible, then  $\xi_{can}$ is {\em hypertight}, that is, it
can be defined by a contact form whose associated Reeb vector field has no
contractible orbits. Thus, for example when $\mathcal{T} \neq \emptyset$, $\xi_{can}$
is hypertight. Note that hypertight contact structures are tight \cite{hof} and any finite cover of 
a hypertight contact manifold is hypertight \cite{gh}. 
These results together with the fact that graph manifolds have residually finite
fundamental groups give another proof of universally tightness (avoiding Colin's
gluing result). Since $M$ is irreducible, its universal cover is diffeomorphic to
either $S^3$ or $\mathbb{R}^3$ depending on whether $\pi_1(M)$ is finite or infinite.
The universal cover is $S^3$ if and only if $M$ is atoroidal, then $M$ is either a small
Seifert fibered space or a lens space and these have no hypertight contact structures.
Therefore, $M$ is hypertight if and only if $\pi_1(M)$ is infinite (or equivalently
its universal cover is $\mathbb{R}^3$).

\end{remark}

\begin{remark}

It is known that any finite cover of a singularity link is a singularity link.
Therefore, another approach to prove Theorem \ref{unit} would be to show that a finite
cover of a Milnor fillable contact structure is Milnor fillable. It is not clear to
the authors of this paper whether this is indeed true. Note that there exist finite
covers of Stein fillable contact structures which are not tight (in particular, not Stein
fillable) \cite{gompf}.  

\end{remark}

\begin{remark} Since any Milnor fillable contact $3$-manifold $(Y,\xi)$ is Stein
fillable (see \cite{bd}) , it follows from Theorem 1.5 in \cite{OS} that the
contact invariant $c(\xi) \in \widehat{HF}(-Y)/(\pm 1)$ is non-trivial. Therefore, by
\cite{GHV}, the Giroux torsion of $Y$ is zero. In particular, the incompressible
tori in $\mathcal{T}$ have zero torsion. This was predicted in \cite{nem} and was
raised as a question there.
\end{remark}

\section{Universally tight but no taut}

A rational homology sphere is called an $L$-space if $\rank \widehat{HF}(Y) = |H_1(Y; \bfz)|$. Lens spaces are basic examples of $L$-spaces which explains the name. A characterization of $L$-spaces among Seifert fibered $3$-manifolds is given by

\begin{theorem}\label{hom} \cite{ls} A rational homology sphere which is Seifert fibered over $S^2$  is an $L$-space if and only if  it does not carry a taut foliation. 
 \end{theorem}

A huge class of examples of $L$-spaces come from complex surface singularities. Recall
that an isolated normal surface singularity $(X,x)$ is rational (cf. \cite{a}) if the geometric
genus $p_g := \text{dim}_\mathbb{C} H^1(\tilde{X} , \mathcal{O}_{\tilde{X}})$ is
equal to zero, where $\tilde{X} \to X$ is a resolution of the singular point $x\in X$. This definition does not depend on the resolution.

\begin{theorem}\label{rat} \cite{nem} The link of a rational surface singularity is an $L$-space. 

\end{theorem}

\begin{corollary}\label{main} 
 
If $Y$ is the link of a rational surface singularity which is Seifert fibered over $S^2$, then $Y$ carries a universally tight contact structure that can not be obtained by a deformation of a taut foliation. 

\end{corollary}

\begin{proof}
The link of a rational surface singularity is an $L$-space by Theorem~\ref{rat} and hence it does not carry any taut foliations by Theorem~\ref{hom}. Moreover,  Theorem~\ref{unit} implies that the canonical contact structure on this link is universally tight. \end{proof}

\begin{remark} \label{ibne}

Note that Seifert fibered $3$-manifolds as above carry transverse contact structures (by Theorem 
$1.3$ in   \cite{lm}) and such contact structures are known to be universally tight (cf. Corollary $2.2$ in \cite{lm} and also Proposition $4.4$ in \cite{mas}).
\end{remark} 

\begin{corollary}\label{app} There exist infinitely many atoroidal $3$-manifolds with
infinite fundamental groups which carry universally tight contact structures that are
not deformations of taut (or Reebless) foliations.

\end{corollary}

\begin{proof} 

It is known (cf. \cite{dim}) that the link of a complex surface singularity has finite
fundamental group if and only if it is a quotient singularity. Thus the link of a
rational but not quotient surface singularity has an infinite fundamental group. Note
that the links of a quotient surface singularities (all small Seifert fibered
$3$-manifolds) are explicitly listed in \cite{bho} via their dual resolution graphs.
It is easy to see that there are many infinite families of small Seifert fibered
$3$-manifolds which are links of rational but not quotient surface singularities. This
finishes the proof using Corollary~\ref{main} since all small Seifert fibered
$3$-manifolds are known to be atoroidal.  Note that on an atoroidal $3$-manifold, a Reebless foliation is taut.
\end{proof} 

Consequently,
Corollary~\ref{app} answers Questions 4 and 5 of Etnyre \cite{etn} negatively. For the sake of
completeness we give an infinite family of counterexamples. The small Seifert fibered
$3$-manifold  $$Y_p= Y(-2; \frac{1}{3}, \frac{2}{3}, \frac{p}{p+1})$$ can be described by the surgery 
diagram depicted in Figure~\ref{seif}, where $p$ is a positive integer.
Note that $Y_p$ is the link of a complex surface singularity whose dual resolution graph is given in Figure~\ref{yprat}. 
\begin{figure}[ht]
  \relabelbox \small {\epsfxsize=2.5in
  
\centerline{\epsfbox{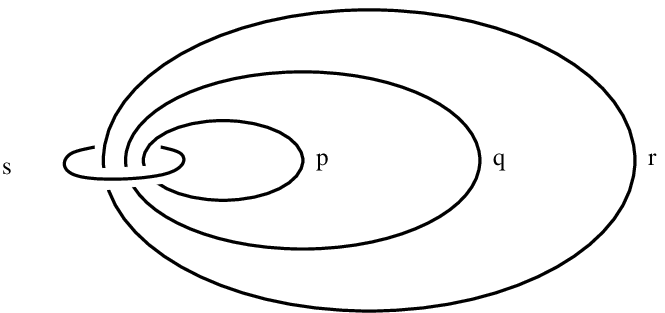}}}

\relabel{p}{$-3$}

\relabel{q}{$-\frac{3}{2}$}

\relabel{r}{$-\frac{p+1}{p}$}

\relabel{s}{$-2$}

\endrelabelbox
        \caption{Rational surgery diagram for $Y_p$ } \label{seif}
 \end{figure}

\begin{figure}[ht]

  \relabelbox \small {\epsfxsize=2.5in
  \centerline{\epsfbox{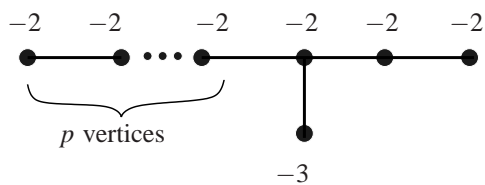}}}

\relabel{a}{$-2$}

\relabel{b}{$-2$}

\relabel{c}{$-2$}

\relabel{d}{$-3$}

\relabel{e}{$-2$} \relabel{f}{$-2$} \relabel{g}{$-2$}

\relabel{r}{$p$ vertices}

\endrelabelbox
           \caption{Dual resolution graph} \label{yprat}
 \end{figure}

Let $(X,x)$ be a germ of a complex surface singularity. Fix a resolution
$\pi\colon \tilde X\to X$ and denote the irreducible
components of the exceptional divisor $E=\pi^{-1}(x)$ by
$\bigcup_{i=1}^n E_i$. The {\em fundamental cycle} of $E$ is
by definition the componentwise smallest nonzero effective divisor $Z=\sum z_i E_i$
satisfying $Z\cdot E_i\leq 0$ for all $1 \leq i \leq n$.  It turns out that the singularity $(X,x)$
is  rational if each irreducible component $E_i$ of the
exceptional divisor $E$ is isomorphic to $\bfc P^1$ and
$$ Z\cdot Z+\sum_{i=1}^n z_i(-E_i^2-2)= -2, $$
where $Z=\sum z_iE_i$ is the fundamental cycle of $E$.

Enumerate the vertices in the dual resolution graph for $Y_p$ from left to right along the
top row with the bottom vertex coming last (see Figure~\ref{yprat}).  It is then easy to check  (cf. \cite{bo})
that the coefficients $(z_1, z_2, \ldots,  z_n)$ of the corresponding fundamental cycle 
is given by $(1,2,3,3,\ldots,3,3,2,1,1)$.  It follows that  $Y_p$ is
the link of a rational surface singularity and hence it is an 
L-space. We conclude that the
canonical contact structure $\xi_{can} $ on $Y_p$ is universally tight but it can not
be obtained by perturbing  a taut foliation.  Moreover, if $p \geq 2$, then $Y_p$ is not a quotient singularity 
\cite{bho} and thus its fundamental group is infinite.  


\v  \noindent {\bf {Acknowledgement}}: We would like to thank Tolga Etg\"u for helpful
conversations, Patrick Massot for his comments on a draft of this paper and the Mathematical Sciences Research Institute for its hospitality during the \textit{Symplectic and Contact Geometry and Topology} program 2009/2010. B.O. was partially supported by the BIDEP-2219 research grant of the Scientific and Technological
Research Council of Turkey and the Marie Curie International Outgoing Fellowship 236639. 

\bibliographystyle{amsplain}
\providecommand{\bysame}{\leavevmode\hbox
to3em{\hrulefill}\thinspace}
\providecommand{\MR}{\relax\ifhmode\unskip\space\fi MR }

\providecommand{\MRhref}[2]{%
  \href{http://www.ams.org/mathscinet-getitem?mr=#1}{#2}
} \providecommand{\href}[2]{#2}

\begin{flushleft}{\small{
Mathematical Sciences Research Institute, Berkeley, CA 94720 \\ 
\verb"ylekili@msri.org"

Mathematical Sciences Research Institute, Berkeley, CA 94720 \\ 
Department of Mathematics, Ko\c{c} University, Istanbul, Turkey \\
\verb"bozbagci@ku.edu.tr"

}}

\end{flushleft}
\end{document}